\documentclass[12pt]{amsart}%

\usepackage{amssymb}
\usepackage{amsmath}
\usepackage{amsthm}
\usepackage{amscd}
\usepackage[margin=1.0in]{geometry}%
\usepackage{hyperref}

\theoremstyle{plain}
\newtheorem{theorem}{Theorem}[section]

\newtheorem{corollary}{Corollary}[section]
\newtheorem{lemma}{Lemma}[section]

\theoremstyle{remark}
\newtheorem{definition}{Definition}[section]
\newtheorem{remark}{Remark}[section]

\DeclareMathOperator{\lin}{span}

\DeclareMathOperator{\dom}{dom}
\DeclareMathOperator{\Imag}{Im}

\subjclass[2000]{Primary 60J25, 81Q35; Secondary 
05C81, 
28A25, 
31C25, 
46J10, 
47A07, 
47A10, 
53C23,
54E45, 
}

\keywords{Mosco convergence, strong resolvent convergence, lower semicontinuous quadratic forms, Dirichlet forms, countably generated and uniformly closed algebras of bounded functions,  Laplacians on graphs, random walks.}
\begin{document}

\begin{abstract}
We consider a countably generated and uniformly closed algebra of bounded functions. We assume that there is a lower semicontinuous, with respect to the supremum norm, quadratic form and that normal contractions operate in a certain sense. Then we prove that a subspace of the effective domain of the quadratic form is naturally isomorphic to a core of a regular Dirichlet form on a locally compact separable metric space. We also show that any Dirichlet form on a countably generated measure space can be approximated by essentially discrete Dirichlet forms, i.e. energy forms on finite weighted graphs, in the sense of Mosco convergence, i.e. strong resolvent convergence.

\tableofcontents
\end{abstract}

\title{Closability, regularity, and  approximation by graphs for separable bilinear forms}
\author{Michael Hinz$^1$}
\address{$^1$ Fakult\"at f\"ur Mathematik, Universit\"at Bielefeld, Postfach 100131, 33501 Bielefeld, Germany}
\email{mhinz@math.uni-bielefeld.de}

\author{Alexander Teplyaev$^2$}
\address{$^2$Department of Mathematics, University of Connecticut, Storrs, CT 06269-3009 USA}
\email{teplyaev@uconn.edu}

\date{\today}

\maketitle

\section[Introduction and result]{Introduction and results}

The first purpose of the present note is to discuss quadratic forms on countably generated algebras of functions and a way to turn them into (symmetric) Dirichlet forms, \cite{BH91, FOT94, ChF, MR92}. This complements our former studies \cite{H14} and \cite{HKT}. 
The second purpose is to show that Dirichlet forms on countably generated measures spaces can be approximated, in the Mosco sense, by essentially discrete energy forms (essentially equivalent to the energy on finite graphs). This approximation is the same as  the approximation in the strong resolvent sense, \cite{RS}, 
and also  can be described as piece-wise constant (in particular, typically discontinuous) Galerkin approximations. 

Our paper is expository, although Theorems~\ref{T:main} and \ref{T:main2} previously have not appeared in the literature. The long term motivation for such connections comes from the fact that approximations by random walks allows to construct diffusion processes on some universal topological objects, 
\cite{KZ,BBKT, Ki,Str}, which in turn allows to construct notions of differential geometry 
not previously available, \cite{HKT,HRT,HTdms,HThodge}. 
One one hand, these notions may be applicable to complicated spaces defined algebraically, such as in 
\cite{Nekrashevich,NT}. 
On the other hand, there are potential physics applications for such notions, see 
\cite{Englert,Loll,Reuter}. 
We would like to note that random walks and discrete time martingales appear in a number of works of 
M.I.~Gordin, such as
\cite[and references therein]{Gordin,GP}, and our paper may be a step towards continuous versions of these results.

A \emph{Dirichlet form} $(\mathcal{E},\mathcal{F})$ is a symmetric, bilinear and positive definite real valued form $\mathcal{E}$ on a subspace $\mathcal{F}$ that is dense in a real $L^2$-space $L^2(X,\mathcal{X},\mu)$ over a $\sigma$-finite measure space $(X,\mathcal{X},\mu)$ such that 
\begin{itemize}
\item the space $\mathcal{F}$, together with the norm $f\mapsto (\mathcal{E}(f)+\left\|f\right\|_{L^2(X,\mathcal{X},\mu)}^2)^{1/2}$, is a Hilbert space (the '\emph{Dirichlet space}') and 
\item for any $f\in\mathcal{F}$ we have $(f\wedge 1)\vee 0 \in \mathcal{F}$ and $\mathcal{E}(f)\leq \mathcal{E}((f\wedge 1)\vee 0)$.
\end{itemize}
Here we use the notation $\mathcal{E}(f):=\mathcal{E}(f,f)$ which we will also employ for other symmetric bilinear expressions. There is a one-to-one correspondence of Dirichlet forms and non-positive definite self-adjoint operators on $L^2(X,\mathcal{X},\mu)$ satisfying a certain Markov property, see e.g.\cite[Proposition I.3.2.1]{BH91}. The self-adjoint operator $(\mathcal{L},\dom \mathcal{L})$ uniquely associated with $(\mathcal{E},\mathcal{F})$ (and referred to as its \emph{generator}) satisfies
\[\mathcal{E}(f,g)=-\left\langle \mathcal{L}f, g\right\rangle_{L^2(X,\mathcal{X},\mu)}, \ \ f\in \dom \mathcal{L}, \ g\in \mathcal{F},\]
and is uniquely determined by this formula. More background on Dirichlet forms can for instance be found in \cite{BH91, FOT94, MR92}.

A particularly rich theory is available if $X$ is a locally compact separable metric space, $\mathcal{X}$ the Borel-$\sigma$-algebra over $X$ and $\mu$ a nonnegative Radon measure on $X$ with full support. If in this case we can find a core $\mathcal{C}\subset C_c(X)\cap \mathcal{F}$ dense in the Hilbert space $\mathcal{F}$ and dense in $C_c(X)$ with respect to the supremum norm $\left\|\cdot\right\|_{\sup}$, then the Dirichlet form $(\mathcal{E},\mathcal{F})$ is called \emph{regular}. Here $C_c(X)$ denotes the space of continuous compactly supported functions on $X$, see \cite{FOT94}. A regular Dirichlet form is compatible with the topology on $X$, what provides strong connections between the calculus of variations and functional analysis on one hand and potential theory and stochastic processes on the other.

We are interested in starting from an algebra of functions (or even simpler, from a sequence of functions) on which a quadratic functional defined and then to 'create' a related Dirichlet form. This way, the quadratic functional (the '\emph{energy}') has priority over the choice of a reference measure or a topology on $X$. In a sense, this follows the spirit of early literature on Dirichlet forms, such as \cite{Allain} and \cite{And75}. Related (but different) discussions of representation theory for Dirichlet forms can be found in \cite{C06} and \cite[Appendix A.4]{FOT94}.

In \cite{HKT} we sketched how a given Dirichlet form on a measure space can be transferred into a Dirichlet form on a locally compact Hausdorff space. The latter is given by the Gelfand spectrum of an algebra obtained from bounded measurable functions of finite energy. (We would like to point out that unfortunately \cite{HKT} contained two inaccuracies: The transfer of an infinite measure  to the Gelfand spectrum needs to be corrected using a unitization procedure, \cite{Kaniuth}, and for the transferred form to be regular the given measure space must have a countably generated $\sigma$-algebra.) In \cite{H14} we considered real valued bilinear forms on algebras of bounded functions and used energy measures to obtain a Dirichlet form. We did not assume the algebra to be countably generated. Key ingredients were measure-free notions of closability and lower semicontinuity, such as the following.

\begin{definition}
Let $X$ be a nonempty set and $A$ a space of bounded real valued functions on $X$ which is complete with respect to the supremum norm. A quadratic functional $E:A\to [0+\infty]$ is called \emph{sup-norm-lower semicontinuous on $A$} if for any sequence $(f_n)_n\subset A$ with uniform limit $f\in A$ we have 
\[E(f)\leq \liminf_{n} E(f_n).\]
\end{definition}

In the present note we concentrate on notions of \emph{separability}. We start from a countably generated and uniformly closed algebra $\mathcal{A}$ of bounded real valued functions and a quadratic form $\mathcal{E}$ on $\mathcal{A}$, a priori extended real valued.

\begin{definition}
Let $X$ be a nonempty set and $\mathcal{A}$ a countably generated uniformly closed algebra of bounded real valued functions on $X$ that separates the points of $X$. A quadratic form $\mathcal{E}:\mathcal{A}\to [0,+\infty]$ is called a \emph{sup-norm-separable quadratic form} if there is a sequence $\left\lbrace f_n\right\rbrace_{n=1}^\infty\subset \mathcal{A}$ generating $\mathcal{A}$ such that $\mathcal{E}(f_n)<+\infty$ for all $n$. 
\end{definition}

Throughout the following let $\mathcal{A}$ be as in the definition. Since $\mathcal{A}$ is countably generated and uniformly closed, it is separable.

\begin{remark}
If $X$ itself is a locally compact separable metric space and $\mathcal{A}$ is an algebra of continuous functions that (in addition to point separation) vanishes nowhere on $X$, then by Stone-Weierstrass $\mathcal{A}=C_0(X)$ is the algebra of continuous functions vanishing at infinity. If $(\mathcal{E},\mathcal{F})$ is a regular Dirichlet form on $X$ and we consider $\mathcal{E}$ as an extended real valued quadratic form $C_0(X)$, then it is sup-norm-separable. Note that $\mathcal{E}$ is sup-norm-lower semicontinuous on $C_c(X)\cap \mathcal{F}$, see \cite[Theorem 2.1 and Theorem 10.1]{H14}.
\end{remark}

Restricted to its effective domain a quadratic form yields a symmetric bilinear form that can be transferred to the Gelfand spectrum $\Delta$ of the natural complexification of $\mathcal{A}$. It is a locally compact Hausdorff space, and by the separability of $\mathcal{A}$ it is second countable and metrizable. Proceeding as in \cite[Section 9]{H14} we can then extend the transferred form to a regular Dirichlet form on an $L^2$-space over the Gelfand spectrum $\Delta$. 

To formulate this result we fix two more notions. A quadratic form $\mathcal{E}$ on an algebra $\mathcal{A}$ of bounded real valued functions is called \emph{$\sigma$-finite} if $\mathcal{A}$ contains a strictly positive function $\chi$ such that $\mathcal{E}(\chi)<+\infty$. To a function $F:\mathbb{R}^k\to\mathbb{R}$ such that 
\[|F(x)-F(y)|\leq \sum_{i=1}^k|x_i-y_i|, \ \  x,y\in\mathbb{R}^k,\] 
and $F(0)=0$ we refer as a \emph{normal contraction} (in several variables). If the quadratic form $\mathcal{E}$ satisfies
\begin{equation}\label{E:normalcont}
\mathcal{E}(F(f_1,...,f_k))^{1/2}\leq \sum_{i=1}^k\mathcal{E}(f_i)^{1/2}
\end{equation}
for all $f_1,...,f_k\in\mathcal{A}$ and all normal contractions $F$, then we say that \emph{normal contractions} operate on $\mathcal{E}$, cf. \cite[Section I.3.3]{BH91}. 

\begin{theorem}\label{T:main}
Let $X$ be a nonempty set and let $\mathcal{A}$ be a countably generated and uniformly closed algebra $\mathcal{A}$ of bounded real valued-functions on $X$ that separates the points of $X$. 
Let
$\mathcal{E}$ be a sup-norm-separable, $\sigma$-finite and sup-norm-lower semicontinuous quadratic form on $\mathcal{A}$ on which normal contractions operate. Then the effective domain 
\[\mathcal{D}=\left\lbrace f\in\mathcal{A}: \mathcal{E}(f)<+\infty\right\rbrace\] 
of $\mathcal{E}$ is isomorphic to an algebra $\hat{\mathcal{D}}$ of continuous functions in the domain of a regular Dirichlet form $\hat{\mathcal{E}}$ on a locally compact separable metric space and $\hat{\mathcal{D}}$ contains a core for $\hat{\mathcal{E}}$. 
The isomorphism $f\mapsto \hat{f}$ from $\mathcal{D}$ onto $\hat{\mathcal{D}}$ preserves the algebraic structure, the supremum norm, and is also  'isometric' in the sense that $\hat{\mathcal{E}}(\hat{f})=\mathcal{E}(f)$, $f\in\mathcal{D}$.
\end{theorem}

Both Theorem \ref{T:main} and the results in \cite{H14} should be seen as continuations of and complements to an earlier study of Mokobodzki, \cite{Mok95}, where sup-norm-lower semicontinuity was employed to conclude the closability of quadratic forms on spaces of continuous functions over locally compact spaces. \\

A second aim of the present note is to provide approximations of Dirichlet forms in terms of discrete energy forms on finite graphs, and this approximation is in terms of a suitable 'spectral convergence'. We recall the definition of convergence in strong resolvent sense, \cite[Chapter VIII.7]{RS}. Usually it is formulated for complex Hilbert spaces, we consider it with respect to the natural complexification of $L^2(X,\mathcal{X},\mu)$. A sequence $(L^{(n)})_{n=1}^\infty$ of self-adjoint operators $L^{(n)}$ on $L^2(X,\mathcal{X},\mu)$ with resolvents $G_\lambda^{(n)}=(\lambda-L^{(n)})^{-1}$ \emph{converges in the strong resolvent sense} to a self-adjoint operator $L$ on $L^2(X,\mathcal{X},\mu)$ with resolvent $G_\lambda=(\lambda-L)^{-1}$ if for all $\lambda \in\mathbb{C}$ with $\Imag \lambda\neq 0$ the resolvent operators $G_\lambda^{(n)}$ converge to $G_\lambda$ in the strong operator topology. This condition is equivalent to requiring the convergence for some $\lambda \in\mathbb{C}$ with $\Imag \lambda\neq 0$, \cite[Theorem VIII.19]{RS}. In the case of non-positive definite self-adjoint operators $L^{(n)}$ and $L$ that generate strongly continuous contraction semigroups it is also equivalent to requiring the convergence of the resolvents $G_\lambda^{(n)}$ to $G_\lambda$ in the strong operator topology for some (and hence all) $\lambda>0$, see \cite[Theorem 1.3 in Chapter 8, Paragraph 1, Section 1, p. 427]{Kato}. This in turn is equivalent to the Mosco convergence of the associated quadratic forms, \cite[Definition 2.1.1 and Theorem 2.4.1]{Mosco94}. Recall that a sequence $(E^{(n)})_{n=1}^\infty$ of (possibly extended real valued) quadratic forms $E^{(n)}$ on $L^2(X,\mathcal{X},\mu)$ \emph{converges} to a 
quadratic form $E$ on $L^2(X,\mathcal{X},\mu)$ \emph{in the sense of Mosco} if

\begin{enumerate}
\item[(i)] for any sequence $(f_n)_{n=1}^\infty\subset L^2(X,\mathcal{X},\mu)$ converging to some $f$ weakly in $L^2(X,\mathcal{X},\mu)$ we have
\[E(f)\leq \liminf_n E^{(n)}(f_n)\]
and
\item[(ii)] for any $f\in L^2(X,\mathcal{X},\mu)$ there exists a sequence $(f_n)_{n=1}^\infty$ converging to $f$ strongly in $L^2(X,\mathcal{X},\mu)$ and such that
\[\limsup_n E^{(n)}(f_n)\leq E(f).\]
\end{enumerate} 

Again a main ingredient is a suitable notion of separability.

\begin{definition}
To a Dirichlet form $(\mathcal{E},\mathcal{F})$ on a space $L^2(X,\mathcal{X},\mu)$ over a
$\sigma$-finite measure space $(X,\mathcal{X},\mu)$ with a countably generated $\sigma$-algebra $\mathcal{X}$ we refer as a \emph{separable Dirichlet form}.
\end{definition}

\begin{remark}
Any regular Dirichlet form in the sense of \cite{FOT94} is separable.
\end{remark}

To a Dirichlet form which (in the sense of its Dirichlet space) is essentially isometrically isomorphic to a discrete Dirichlet ('energy') form on a finite graph we simply refer as a \emph{essentially discrete Dirichlet form}. 
Technically speaking our approximating forms are infinite dimensional,
but the generators have finite dimensional range (but typically
infinite dimensional kernel). Using this manner of speaking, our second result reads as follows. 

\begin{theorem}\label{T:main2}
Any separable Dirichlet form $(\mathcal{E},\mathcal{F})$ 
can be approximated in the Mosco
sense by a sequence of essentially discrete Dirichlet forms
(essentially isomorphic to that on finite weighted graphs) and the
corresponding generators approximate the generator of $(\mathcal{E},\mathcal{F})$ in the strong resolvent sense.
\end{theorem}

Our paper is organized as follows. 
We will first review some basics on algebras of functions and the Gelfand transform in Section \ref{S:Alg}, then consider quadratic forms in Section \ref{S:Quad} and finally, employ results from \cite{H14} to conclude the existence of the regular Dirichlet form on the Gelfand spectrum in Corollary \ref{C:final} of Section \ref{S:Clos}, what proves Theorem \ref{T:main}. To prove Theorem \ref{T:main2} we discuss the isomorphy of Dirichlet forms on finitely generated measure spaces to graph energies in Section \ref{S:finite}, semigroup approximation is Section \ref{S:sgapprox} and further approximations in terms of Galerkin schemes, $\sigma$-finiteness and level sets of base elements in Section \ref{S:Galerkin}.

\subsection*{Acknowledgments}
The second author thanks Robert Strichartz for raising the question about approximating 
arbitrary Laplacians by graph Laplacians. 

\section[Algebras of functions]{Algebras of functions}\label{S:Alg}

We review some folklore facts about algebras of functions and associated locally compact spaces. Let $X$ be a nonempty set. Given a uniformly closed and countably generated algebra $\mathcal{A}$ of bounded real-valued functions on $X$, let $\mathcal{A}_{\mathbb{C}}$ denote its natural complexification and $\Delta$  the \emph{Gelfand spectrum} of $\mathcal{A}_{\mathbb{C}}$, i.e. the space of nonzero complex valued multiplicative linear functionals on $\mathcal{A}_{\mathbb{C}}$, \cite{Blackadar, Kaniuth}. The space $\Delta$ is a locally compact Hausdorff space and $\mathcal{A}$ being separable, the space $\Delta$ is second countable and metrizable.  We assume that $\mathcal{A}$ separates the points of $X$. Then the image $\iota(X)$ of $X$ under $\iota:X\to\Delta$ given by $\iota(x)(f):=f(x)$, is densely embedded in $\Delta$. 

\begin{remark}
If $X$ itself is a locally compact Hausdorff space and $\mathcal{A}=C_0(X)$ then $\iota$ is a homeomorphism from $X$ onto $\Delta$. If $X$ is a locally compact separable metric space, then $C_0(X)$ is separable, hence $\Delta$ is metrizable.
\end{remark}

We collect these well known arguments in the following Lemma.

\begin{lemma}\label{L:algebra}
Let $\mathcal{A}$ be a uniformly closed and countably generated algebra of bounded real valued functions on $X$. Assume that $\mathcal{A}$ separates the points of $X$. Then the Gelfand spectrum $\Delta$ of the natural complexification of $\mathcal{A}$ is a locally compact separable metric space, and $X$ may be regarded as a dense subset.
\end{lemma}

\begin{remark}
The topology of $\Delta$ is compatible with the uniform structure determined by the sets of form $U_{F,\varepsilon}=\left\lbrace (x,y)\in \Delta\times \Delta: |f(x)-f(y)|<\varepsilon, \ \text{ for all $f\in F$}\right\rbrace$,
where $\varepsilon>0$ and $F$ ranges over all finite subfamilies of functions from $\mathcal{A}_\mathbb{C}$. See \cite[Chapter II, Section 1.1]{Bourbaki}. If $\left\lbrace f_n\right\rbrace_{n=1}^\infty$ is a sequence 
of bounded real valued functions that generates $\mathcal{A}$ then since $\mathcal{A}$ separates the point of $X$ also $\left\lbrace f_n\right\rbrace_{n=1}^\infty$
separates points. Let $\mathcal{P}$ denote the collection of polynomials in several variables of the functions $\left\lbrace f_n\right\rbrace_{n=1}^\infty$ and having coefficients with rational real and imaginary parts. The space $\mathcal{P}$ is a countable and uniformly dense subalgebra of $\mathcal{A}_{\mathbb{C}}$. The functions
$\varrho_F(x,y):=\sup\left\lbrace |f(x)-f(y)|: f\in F\right\rbrace$, $x,y\in \Delta$,
where $F$ ranges over all finite subfamilies of functions from $\mathcal{P}$, provide a countable family of pseudometrics on $\Delta$ that generate this uniform structure. Now 
\cite[Chapter IX, Section 2.4, Corollary 1 and Section 1.4, Proposition 2]{Bourbaki} give another proof of the known fact that $\Delta$ is metrizable. A straightforward choice for a possible metric is 
\[\varrho(x,y):= \sum_{i=1}^\infty 2^{-i}\frac{|f_i(x)-f_i(y)|}{1+|f_i(x)-f_i(y)|}, \ \ x,y\in \Delta.\]
\end{remark}

For any $f\in\mathcal{A}_{\mathbb{C}}$ the \emph{Gelfand transform} $\hat{f}:\Delta\to\mathbb{C}$ of $f$ is defined by $\hat{f}(\varphi):=\varphi(f)$, $\varphi\in\Delta$. According to the Gelfand representation theorem the map $f\mapsto \hat{f}$ defines an $^\ast$-isomorphism from $\mathcal{A}_\mathbb{C}$ onto the algebra $C_{\mathbb{C},0}(\Delta)$ of complex valued continuous functions on $\Delta$ that vanish at infinity. For more background see for instance \cite{Blackadar} or \cite{Kaniuth}.

\begin{remark}
If $X$ is a locally compact Hausdorff space, $A=C_0(X)$, and we identify the spaces $X$ and $\Delta$ under $\iota$, then $f\mapsto \hat{f}$ is the identity mapping.
\end{remark}

\section[Quadratic forms]{Quadratic forms}\label{S:Quad}

We consider quadratic forms on $\mathcal{A}$ and see how to transform them into quadratic forms on spaces of continuous functions on $\Delta$. 

Let $\mathcal{A}$ be as in Lemma \ref{L:algebra}, i.e. a uniformly closed and countably generated algebra of bounded real valued functions on $X$, which is point separating. Assume that we are given a (extended real valued) quadratic form
\[\mathcal{E}:\mathcal{A}\to [0,+\infty]\]
and write 
\[\mathcal{D}:=\left\lbrace f\in\mathcal{A}: \mathcal{E}(f)<+\infty\right\rbrace\]
for its effective domain. On $\mathcal{D}$ we can define a (real-valued) bilinear form $(\mathcal{E},\mathcal{D})$ by polarization,
\[\mathcal{E}(f,g):=\frac14\left(\mathcal{E}(f+g)-\mathcal{E}(f-g)\right), \ \ f,g\in\mathcal{D}.\]
If normal contractions operate on $\mathcal{E}$ then the effective domain $\mathcal{D}$ is stable under normal contractions, i.e. if $F:\mathbb{R}^k\to\mathbb{R}$ is a normal contraction and $f_1,...,f_k \in\mathcal{D}$ then we have $F(f_1,..., f_k)\in\mathcal{D}$. In particular,
\[\mathcal{E}(fg)^{1/2}\leq \left\|f\right\|_{\sup}\mathcal{E}(g)^{1/2}+\left\|g\right\|_{\sup}\mathcal{E}(f)^{1/2}\]
for all $f,g\in \mathcal{A}$, and $\mathcal{D}$ is an algebra under pointwise multiplication. See the proof of \cite[Corollary I.3.3.2]{BH91}.

If in addition there is a sequence $\left\lbrace f_n\right\rbrace_{n=1}^\infty$ generating $\mathcal{A}$ (and therefore point separating) such that $\mathcal{E}(f_n)<+\infty$ for all its members $f_n$, then $\mathcal{D}$ is a uniformly dense subalgebra of  $\mathcal{A}$ and the algebra
\[\hat{\mathcal{D}}:=\left\lbrace \hat{f}\in C(\Delta):f\in\mathcal{D}\right\rbrace\] 
is uniformly dense in the subalgebra $C_0(\Delta)$ of real valued continuous functions on $\Delta$ vanishing at infinity. By
\[\hat{\mathcal{E}}(\hat{f},\hat{g}):=\mathcal{E}(f,g), \ \ \hat{f},\hat{g}\in\hat{\mathcal{D}}\]
we can define a bilinear form $\hat{\mathcal{E}}$ on $\hat{\mathcal{D}}$ (the \emph{transferred form}). See \cite[Section 10]{H14} or \cite{HKT} for further information.

\section[Sup-norm-lower semicontinuity and closability]{Sup-norm-lower semicontinuity and closability}\label{S:Clos}

We employ the notions of sup-norm-lower semicontinuity and sup-norm closability to deduce the closability and the contractivity for the transferred form.

\begin{definition}
Let $X$ be a nonempty set and $D$ a space of bounded real valued functions on $X$. A (real-valued) bilinear form $(E,D)$ is called \emph{sup-norm-closable} if for any sequence $(f_n)_n\subset D$ which is $\mathcal{E}$-Cauchy and such that $\lim_n \left\|f_n\right\|_{\sup}=0$ we have
\[\lim_n E(f_n)=0.\]
\end{definition}

Now let $\mathcal{A}$, $\mathcal{E}$, $\hat{\mathcal{E}}$, $\mathcal{D}$ and $\hat{\mathcal{D}}$ be as in the preceding section. The following is an immediate consequence of \cite[Corollary 10.1]{H14}.

\begin{lemma}\label{L:supnormclos}
Suppose $\mathcal{E}$ is sup-norm-lower semicontinuous on $\mathcal{A}$. Then both $(\mathcal{E},\mathcal{D})$ and $(\hat{\mathcal{E}},\hat{\mathcal{D}})$ are sup-norm-closable. Moreover, $\hat{\mathcal{E}}$ is sup-norm-lower semicontinuous on $\hat{\mathcal{D}}$. 
\end{lemma}

In the presence of sup-norm-lower semicontinuity (resp. sup-norm-closability), contractivity properties carry over. 

\begin{lemma}
Suppose $\mathcal{E}$ is sup-norm-lower semicontinuous on $\mathcal{A}$ and that
normal contractions operate on $\mathcal{E}$. Then normal contractions operate on $(\hat{\mathcal{E}},\hat{\mathcal{D}})$, i.e. (\ref{E:normalcont}) holds for all normal contractions $F$ but with 
$\hat{\mathcal{E}}$ and $\hat{f}_1,...,\hat{f}_k\in\hat{\mathcal{D}}$ in place of $\mathcal{E}$ and $f_1,...,f_k\in\mathcal{D}$. 
\end{lemma}

The lemma follows by noting that on compact subsets of $\mathbb{R}^k$ a normal contraction $F:\mathbb{R}^k\to\mathbb{R}$ can by approximated in $C^1$-norm by polynomials $P_n$ and we have $P_n(\hat{f}_1,...,\hat{f}_k)=(P_n(f_1,...,f_k))^\wedge$.

A second consequence of sup-norm-lower semicontinuity (resp. sup-norm-clo\-sa\-bi\-lity) is that the transferred form extends to a Dirichlet form. In the present separable setup we can make sure this extension is a regular Dirichlet form in the sense of Fukushima et al, \cite{FOT94}. The result follows from \cite[Section 9, in particular Theorem 9.1]{H14}. In a topological setup an earlier version of this theorem was given by Mokobodzki, \cite{Mok95}.

\begin{corollary}\label{C:final}
Let $\mathcal{A}$ be a countably generated uniformly closed algebra $\mathcal{A}$ of bounded real valued-functions on $X$ that vanishes nowhere and separates the points of $X$. Let $\mathcal{E}$ be a sup-norm-separable, $\sigma$-finite and sup-norm-lower semicontinuous quadratic form on $\mathcal{A}$ on which normal contractions operate. 
Then there exists a finite Radon measure $\hat{m}$ on $\Delta$ with full support such that $(\hat{\mathcal{E}},\hat{\mathcal{D}})$ extends to a regular Dirichlet form $(\hat{\mathcal{E}},\hat{\mathcal{F}})$ on $L^2(\Delta,\hat{m})$, and  
\[\hat{\mathcal{D}}_c:=\left\lbrace \hat{f}\in C_c(\Delta): f\in\mathcal{D}\right\rbrace\subset\hat{\mathcal{D}}\] 
is a core for $(\hat{\mathcal{E}},\hat{\mathcal{F}})$.
\end{corollary}

Corollary \ref{C:final} follows from \cite[Theorem 9.1]{H14} together with the fact that $\Delta$ is separable. Note that if the finite (energy dominant) Radon measure constructed in \cite[Theorem 9.1]{H14} does not have full support, we can put $\hat{m}$ to be the sum of this measure and a finite measure of form $\sum_{i=1}^\infty 2^{-i}\delta_{x_i}$, where $\left\lbrace x_i\right\rbrace_{i=1}^\infty$ is a countable dense subset of $\Delta$ and $\delta_{x_i}$ are unit point masses at the $x_i$. Then the finite (energy dominant) Radon measure $\hat{m}$ has full support. That $\hat{\mathcal{D}}_c$ is a core for $(\hat{\mathcal{E}}, \hat{\mathcal{F}})$ can be seen as in \cite[Lemma 3.4]{HKT} and the proof of \cite[Theorem 5.1]{HKT}.

\begin{remark}
The Dirichlet form $(\hat{\mathcal{E}},\hat{\mathcal{F}})$ admits a carr\'e du champ, \cite[Chapter I.4]{BH91}, see for instance \cite{H14}.
\end{remark}

\section{Finitely generated measure spaces and weighted graphs}\label{S:finite}

We discuss bounded Dirichlet forms on finitely generated measure spaces and identify the with discrete Dirichlet forms on finite (weighted) graphs.

Let $X$ be a nonempty set and $\mathcal{X}$ a finitely generated algebra of subsets of $X$. Being finite, it is a $\sigma$-algebra. We may assume there exists a finite partition $\left\lbrace A_1, ..., A_n\right\rbrace$ of $X$ into pairwise disjoint subsets $A_1,... , A_n$. In this case we observe that  
$\mathcal{X}$ consists of all unions $\bigcup_{j=1}^k A_{i_j}$ where $1\leq k\leq n$ and  $i_j\in\left\lbrace 1, ..., n\right\rbrace$ for all $j=1,..., k$. Let $\mu$ be a finite and finitely additive measure on $\mathcal{X}$, then automatically $\sigma$-additive. The space $L^2(X,\mathcal{X}, \mu)$ of $\mu$-square integrable functions equals the finite vector space 
\[\left\lbrace \sum_{i=1}^n \alpha_i\mathbf{1}_{A_i}: \alpha_i\in\mathbb{R}\right\rbrace ,\]
endowed with the scalar product defined by $\left\langle f,g\right\rangle=\sum_{i=1}^n \alpha_i\beta_i \mu(A_i)$ for $f=\sum_i \alpha_i\mathbf{1}_{A_i}$ and $g=\sum_i \beta_i \mathbf{1}_{A_i}$. 

A symmetric Markov operator $P$ on $L^2(X,\mathcal{X},\mu)$ produces a symmetric $(n\times n)$-matrix $C=(c_{ij})_{i,j=1}^n$ by
\[c_{ij}:=\left\langle P\mathbf{1}_{A_i},\mathbf{1}_{A_j}\right\rangle_{L^2(X,\mathcal{X},\mu)}.\]
Consider the quadratic form 
\[\mathcal{E}_P(f):=\left\langle f-Pf,f\right\rangle_{L^2(X,\mathcal{X},\mu)}\]
on $L^2(X,\mathcal{X},\mu)$ induced by $P$. If in addition we assume $P\mathbf{1}=\mathbf{1}$ then we have $\mu(A_j)=\sum_i c_{ij}$ for all $j$ and for a function $f=\sum_{i=1}^n\alpha_i\mathbf{1}_{A_i}$ we easily see that
\[\mathcal{E}_P(f)=\frac12\sum_{i=1}^n\sum_{j=1}^n (\alpha_i-\alpha_j)^2 c_{ij}.\]
A slightly different formula is still true if we allow $P\mathbf{1}<\mathbf{1}$, see \cite[Lemma I.2.3.2.1]{BH91}.

We can identify the partition $\mathcal{P}$ with a finite set $V=\left\lbrace p_1,... p_n\right\rbrace$ and $\mu$ with a finite nonnegative function $\mu$ on $V$ given by $\mu(p_i):=\mu(A_i)$. Then \emph{the space $L^2(X,\mathcal{X},\mu)$ is isometrically isomorphic to the space $l(V)$ of real valued functions on $V$}, and we can \emph{identify a function from  $L^2(X,\mathcal{X},\mu)$ with a function in $l(V)$}. Let $E$ be the collection of elements $(p_i,p_j)$ of $V\times V$ such that $c_{ij}>0$. This yields a finite graph $(V,E)$, endowed with vertex weights $\mu(p_i)$ and edge weights $c_{ij}$. The matrix $C=(c_{ij})_{i,j=1}^n$ is the transition matrix of a symmetric random walk on $V$ and the  energy form $\mathcal{E}_P$ may be identified with the graph energy form on $l(V)$, given by
\[\mathcal{E}_P(f)=\frac12\sum_{i=1}^n\sum_{j=1}^n (f(p_i)-f(p_j))^2 c_{ij}, \ \ f\in l(V),\]
provided $P\mathbf{1}=\mathbf{1}$. If $P\mathbf{1}<\mathbf{1}$ we can still identify $\mathcal{E}_P$ with a (slightly different) graph energy form.

\section[Semigroup approximation and strong resolvent convergence]{Semigroup approximation and strong resolvent convergence}\label{S:sgapprox}

In this section we discuss the well known fact that any Dirichlet form can be approximated by bounded Dirichlet forms.

Let $(X,\mathcal{X},\mu)$ be a $\sigma$-finite measure space and $(\mathcal{E},\mathcal{F})$ a (symmetric) Dirichlet form on $L^2(X,\mathcal{X},\mu)$. Let $(P_t)_{t\geq 0}$ denote the uniquely associated strongly continuous semigroup of symmetric contraction operators $P_t$ on $L^2(X,\mathcal{X},\mu)$. The operators $P_t$ enjoy the Markov property, i.e. $0\leq P_tf\leq 1$ $\mu$-a.e. whenever $0\leq f\leq 1$ $\mu$-a.e. We have
\begin{equation}\label{E:sgapprox}
\mathcal{E}(f)=\lim_{t\to 0}\frac{1}{t}\left\langle f-P_tf,f\right\rangle_{L^2(X,\mathcal{X},\mu)}=\sup_{t>0}\frac{1}{t}\left\langle f-P_tf,f\right\rangle_{L^2(X,\mathcal{X},\mu)}
\end{equation}
for any $f\in \mathcal{F}$, note that the function $t\mapsto \frac{1}{t}\left\langle f-P_tf,f\right\rangle_{L^2(X,\mathcal{X},\mu)}$ is increasing as $t$ decreases to zero (what follows from the spectral theorem). By (\ref{E:sgapprox}) we may extend the definition of $\mathcal{E}$, seen as a quadratic functional with effective domain $\mathcal{F}$, to all of $L^2(X,\mathcal{X},\mu)$. \\

Of course formula (\ref{E:sgapprox}) is standard. Here we use it as a first step in an approximation procedure. For $n=1,2,...$ consider the quadratic forms defined by 
\begin{equation}\label{E:discretesgapprox}
\mathcal{E}^{(n)}(f):=2^n\left\langle f-P_{2^{-n}}f,f\right\rangle_{L^2(X,\mathcal{X},\mu)}, \ \ f\in L^2(X,\mathcal{X},\mu).
\end{equation}
They are bounded Dirichlet forms on $L^2(X,\mathcal{X},\mu)$   satisfying $\mathcal{E}^{(n)}(f)\leq 2^n\left\|f\right\|_{L^2(X,\mathcal{X},\mu)}^2$. Note that the operators $I-P_{2^{-n}}$ satisfy the Markov property and therefore are contractions on $L^2(X,\mathcal{X},\mu)$, \cite[Corollary I.2.2.4]{BH91}. Let $\mathcal{L}$ denote the generator of $(\mathcal{E},\mathcal{F})$ and $\mathcal{L}^{(n)}=2^n(P_{2^{-n}}-I)$ the generators of the forms $\mathcal{E}^{(n)}$, respectively.

\begin{lemma}\label{L:sgapprox}
The forms $\mathcal{E}^{(n)}$ converge to the form $\mathcal{E}$ in the sense of Mosco as $n$ goes to infinity. Their generators $\mathcal{L}^{(n)}$ converge to the generator $\mathcal{L}$ of $\mathcal{E}$ in the strong resolvent sense.
\end{lemma} 

The proof is folklore, we sketch it for convenience.

\begin{proof}
Condition (i) for Mosco convergence is trivially satisfied, because we may assume $f\in\mathcal{F}$ and use the constant sequence $f_n:=f$ together with (\ref{E:sgapprox}). For condition (ii) let $(f_n)_n$ converge weakly to $f$. We may assume that $\liminf_n \mathcal{E}^{(n)}(f_n)<+\infty$. Let $(f_{n_k})_k$ be a subsequence with $\lim_k\mathcal{E}^{(n_k)}(f_{n_k})<+\infty$. As it is uniformly bounded in $L_2(X,\mathcal{X},\mu)$ we can apply the Banach-Saks theorem and extract a subsequence $(f_{n_{k_l}})_l$ converging weakly to $f$ and such that the Ces\`aro averages $\frac{1}{N}\sum_{l=1}^N f_{n_{k_l}}$ converge strongly to $f$. Then for any fixed $n_j$ we have
\begin{multline}
\mathcal{E}^{(n_j)}(f)=\lim_N \mathcal{E}^{(n_j)}\left(\frac{1}{N}\sum_{l=1}^N f_{n_{k_l}}\right)\leq \lim_N \frac{1}{N}\sum_{l=1}^N \mathcal{E}^{(n_j)}(f_{n_{k_l}})\notag\\
\leq \limsup_k \mathcal{E}^{(n_j)}(f_{n_k})\leq \liminf_k \mathcal{E}^{(n_k)}(f_{n_k}).
\end{multline}
As the right hand side does not depend on $j$ and the above holds for any such subsequence $(f_{n_k})_k$, we obtain $\mathcal{E}(f)\leq\liminf_n\mathcal{E}^{(n)}(f_n)$. 
\end{proof}

\section[Galerkin and finite graph approximation]{Galerkin and finite graph approximation}
\label{S:Galerkin}

To prove Theorem \ref{T:main2} we now combine the semigroup approximation from Section \ref{S:sgapprox} with further approximation steps. 

Assume that $\mathcal{X}$ is a \emph{countably generated} $\sigma$-algebra over $X$ and that $\mu$ is a $\sigma$-finite measure on $\mathcal{X}$. Let $(\mathcal{E},\mathcal{F})$ be a symmetric Dirichlet form on $L^2(X,\mathcal{X},\mu)$ and let $\mathcal{E}^{(n)}$ be the bounded Dirichlet forms as defined in (\ref{E:discretesgapprox}). Let $(\varphi_i)_{i=1}^\infty$ be complete orthonormal system in the separable Hilbert space $L^2(X,\mathcal{X},\mu)$ and for any $m$ let $\pi_{m}$ denote the projection onto the finite dimensional subspace $\lin (\varphi_1,...,\varphi_m)$. Given $f\in L^2(X,\mathcal{X},\mu)$ we have $f=\sum_{i=1}^\infty c_i\varphi_i$ with $c_i=\left\langle f,\varphi_i\right\rangle_{L^2(X,\mathcal{X},\mu)}$ and 
\[\pi_m(f)=\sum_{i=1}^m c_i\varphi_i.\]
Clearly $\lim_m\pi_m(f)=f$ in $L^2(X,\mathcal{X},\mu)$. For any $n$ and $m$ consider the bounded Dirichlet forms on $L^2(X,\mathcal{X},\mu)$ defined by 
\begin{equation}\label{E:galerkin}
\mathcal{E}^{(n,m)}(f):=2^n\left\langle \pi_m(f)-P_{2^{-n}}\pi_m(f),\pi_m(f)\right\rangle_{L^2(X,\mathcal{X},\mu)}, \ \ f\in L^2(X,\mathcal{X},\mu).
\end{equation}
Note that $\mathcal{E}^{(n,m)}(f)=\mathcal{E}^n(\pi_m(f))$ and $\mathcal{E}^{(n,m)}(f)\leq 2^n\left\|f\right\|_{L^2(X,\mathcal{X},\mu)}^2$. Let $$\mathcal{L}^{(n,m)}=2^n\pi_m(P_{2^{-n}}-I)\pi_m$$ denote the generator of $\mathcal{E}^{(n,m)}$. The following Lemma provides a second approximation procedure, now of Galerkin type. Its proof is similar to that of Lemma \ref{L:sgapprox}, we omit it.

\begin{lemma}\label{L:galerkin}
For any fixed $n$ the forms $\mathcal{E}^{(n,m)}$ converge to the form $\mathcal{E}^{(n)}$ in the sense of Mosco as $m$ goes to infinity. Their generators $\mathcal{L}^{(n,m)}$ converge to the generator $\mathcal{L}^{(n)}$ of $\mathcal{E}^{(n)}$ in the strong resolvent sense.
\end{lemma} 

In other words, any bounded and separable Dirichlet form can be approximated by Dirichlet forms on finite dimensional $L^2$-spaces.

A third approximation is provided by $\sigma$-finiteness. Let $(X_l)_{l=1}^\infty$ be an increasing sequence of sets $X_l\in\mathcal{X}$ with finite measure $\mu(X_l)<+\infty$ such that $X=\bigcup_{l=1}^\infty X_l$. Given $n$, $m$ and $l$ consider the bounded Dirichlet forms 
\begin{equation}\label{E:sigma}
\mathcal{E}^{(n,m,l)}:=2^n\left\langle \pi_m(f)\mathbf{1}_{X_l}-P_{2^{-n}}(\pi_m(f)\mathbf{1}_{X_l}),\pi_m(f)\mathbf{1}_{X_l}\right\rangle_{L^2(X,\mathcal{X},\mu)}, \ \ f\in L^2(X,\mathcal{X},\mu).
\end{equation}
Their generators are $\mathcal{L}^{(n,m,l)}=2^n\pi_m \mathbf{1}_{X_l}(P_{2^{-n}}-I)\mathbf{1}_{X_l}\pi_m$, respectively. Since 
\[\mathcal{E}^{n,m,l}(f)=\left\|(I-P_{2^{-n}})^{1/2}\pi_m(f)\mathbf{1}_{X_l}\right\|_{L^2(X,\mathcal{X},\mu)}^2\] is increasing in $l$ for any $f$ and $\mathcal{E}^{(n,m)}(f)=\sup_l 
\mathcal{E}^{n,m,l}(f)$, we can proceed as in Lemma \ref{L:sgapprox} to obtain the following.

\begin{lemma}\label{L:sigma}
For any fixed $n$ and $m$ the forms $\mathcal{E}^{(n,m,l)}$ converge to the form $\mathcal{E}^{(n,m)}$ in the sense of Mosco as $l$ goes to infinity. Their generators $\mathcal{L}^{(n,m,l)}$ converge to the generator $\mathcal{L}^{(n,m)}$ of $\mathcal{E}^{(n,m)}$ in the strong resolvent sense.
\end{lemma}

This can be rephrased saying that any separable and bounded Dirichlet form on a finite dimensional $L^2$-space can be approximated by Dirichlet form of similar type with finite volume measure.

We provide a fourth approximation. For any fixed $i,k=1,2,...$ consider the level sets of the function $\varphi_i$ defined by
\[A_{j,k}^{(i)}:=\left\lbrace \frac{j}{2^k}< \varphi_i\leq \frac{j+1}{2^k}\right\rbrace, \ \ \ j=-2^{2k}, ..., 2^{2k}-1,\]
\begin{equation}\label{E:inftysets}
A^{(i)}_{-2^{2k}-1,k}:=\left\lbrace \varphi_i\leq -2^{k}\right\rbrace \ \ \text{ and }\ \ A^{(i)}_{2^{2k},k}:=\left\lbrace 2^k<\varphi_i\right\rbrace.
\end{equation}
They form a finite partition of $X$ consisting of $2^{2k+1}+2$ sets. For convenience we set $A^{(i)}_{j,k}:=\emptyset$ for $j\in\mathbb{Z}$ smaller than $-2^{2k}-1$ or greater that $2^{2k}$. 

\begin{remark}
The partition for $k'\geq k$ is a refinement of the partition for $k$.
\end{remark}

For fixed $m$, we now consider the first $m$ coordinate functions $\varphi_1,...,\varphi_m$ simultaneously and given $\mathbf{j}=(j_1, ..., j_m)$ and $k$, we write
\[\overline{A}_{\mathbf{j},k}^{(m)}:=A_{j_1,k}^{(1)}\cap ...\cap A_{j_m,k}^{(m)}\]
for the intersection of their level sets $A_{j_i,k}^{(i)}$. For any $m$ and $k$ the sets $\overline{A}_{\mathbf{j},k}^{(m)}$, $\mathbf{j}\in\mathbb{Z}^n$, form a partition $\mathcal{P}_{m,k}$ of $X$ and disregarding empty sets, it is a finite partition consisting of at most $(2^{2k+1}+2)^m$ sets.

\begin{remark}
If $k'$ and $m'$ are such that both $k\leq k'$ and $m\leq m'$ then $\mathcal{P}_{m',k'}$ is 
a refinement of $\mathcal{P}_{m,k}$.
\end{remark}

Now let $\mathcal{X}_{m,k}$ denote the finite algebra of subsets of $X$ that is generated by $\mathcal{P}_{m,k}$. It is the sub-$\sigma$-algebra $\mathcal{X}_{m,k}$ of $\mathcal{X}$ made up by all unions of sets from $\mathcal{P}_{m,k}$. We denote the restriction of $\mu$ to $\mathcal{X}_{m,k}$ again by $\mu$. The space $L^2(X,\mathcal{X}_{m,k},\mu)$ of finite linear combinations $\sum_{A\in\mathcal{P}_{m,k}} \alpha_{A} \mathbf{1}_{A}$ is a finite dimensional subspace of $L^2(X,\mathcal{X},\mu)$.

\begin{remark}
If $k'$ and $m'$ are such that both $k\leq k'$ and $m\leq m'$, then we have $\mathcal{X}_{m,k}\subset \mathcal{X}_{m',k'}$  and
$L^2(X,\mathcal{X}_{m,k},\mu)\subset L^2(X,\mathcal{X}_{m',k'},\mu)$.
\end{remark}

Given $l$, write $\mu_l$ for the finite measure given by 
\[\mu_l(A):=\mu(A\cap X_l), \ \ A\in\mathcal{X}.\] 
Let $\mathcal{P}_{m,k}^l$ denote the elements $A$ of $\mathcal{P}_{m,k}$ with positive measure $\mu_l$. Let $\pi_{m,l,k}$ denote the projections into the finite dimensional spaces $L^2(X,\mathcal{X}_{m,k},\mu)$ defined by
\[\pi_{m,l,k}(f):=\sum_{A\in\mathcal{P}_{m,k}^l}\alpha_A^{(m,l,k)}(f)\mathbf{1}_{A\cap X_l}, \ \ \text{ where } \ \ \alpha_A^{(m,l,k)}(f)=\frac{1}{\mu_l(A)}\int_{A}\pi_m(f)d\mu_l, \]
for any $f\in L^2(X,\mathcal{X},\mu)$. Let $\mathcal{P}_{m,k}^{l, \infty}$ denote the collection of elements $A\in \mathcal{P}_{m,k}^l$ that for some $i$ are contained in a set $A^{(i)}_{-2^{2k}-1,k}$ or $A^{(i)}_{2^{2k},k}$ as in (\ref{E:inftysets}).

\begin{lemma}\label{L:approxL2}
For any $m$ and $l$ and any $f\in L^2(X,\mathcal{X},\mu)$ we have 
\[\lim_k\pi_{m,l,k}(f)=\pi_m(f)\mathbf{1}_{X_l}\ \ \text{in $L^2(X,\mathcal{X},\mu)$.}\] 
\end{lemma}

\begin{proof}
Let $m$ be fixed and $\varepsilon>0$. We show that for large enough $k$ (depending on $m$, $l$ and $\varepsilon$) we have
\begin{equation}\label{E:approx}
\left\|\pi_m(f)\mathbf{1}_{X_l}-\pi_{m,l,k}(f)\right\|_{L^2(X,\mathcal{X},\mu)}<\varepsilon.
\end{equation}
Note first that
$$
\int_{X}(\pi_m(f)-\pi_{m,l,k}(f))^2\:d\mu_l =\int_{X}\left(\sum_{A\in\mathcal{P}_{m,k}^l}(\pi_m(f)-\alpha_A^{(m,l,k)}(f))\mathbf{1}_A\right)^2\:d\mu_l$$$$=\sum_{A\in\mathcal{P}_{m,k}^l} \int_{A}\left(\frac{1}{\mu_l(A)}\int_A(\pi_m(f)(x)-\pi_m(f)(y))\mu_l(dy)\right)^2\mu_l(dx).\notag
$$

For any $A\in \mathcal{P}_{m,k}^{l,\infty}$ there is some $i$ such that $A\subset\left\lbrace |\varphi_i|\geq 2^k\right\rbrace$. We have $\mu(\left\lbrace |\varphi_i|\geq 2^k\right\rbrace)\leq 2^{-2k}$ by \v{C}ebyshev's inequality, and therefore
\begin{equation}\label{E:smile}
\mu\left(\bigcup_{A\in\mathcal{P}_{m,k}^{l,\infty}} A\right)\leq \mu\left(\bigcup_{i=1}^m\left\lbrace |\varphi_i|\leq 2^k\right\rbrace\right)\leq \frac{m}{2^{2k}}.
\end{equation}
As a consequence, we see that for sufficiently large $k$,
\[\sum_{A\in\mathcal{P}_{m,k}^{l,\infty}} \int_{A} \pi_m(f)^2d\mu=\int_{\bigcup_{A\in\mathcal{P}_{m,k}^{l,\infty}} A}\pi_m(f)^2 d\mu<\frac{\varepsilon^2}{16},\]
and by Jensen's inequality,
\begin{align}\label{E:firstbound}
\sum_{A\in\mathcal{P}_{m,k}^{l,\infty}} \int_{A} &\left(\frac{1}{\mu_l(A)}\int_A(\pi_m(f)(x)-\pi_m(f)(y))\mu_l(dy)\right)^2\mu_l(dx)\notag\\
&\leq\sum_{A\in\mathcal{P}_{m,k}^{l,\infty}} \int_{A} \frac{1}{\mu_l(A)}\int_A(\pi_m(f)(x)-\pi_m(f)(y))^2\mu_l(dy)\mu_l(dx)\notag\\
&\leq 8 \sum_{A\in\mathcal{P}_{m,k}^{l,\infty}} \int_{A} \pi_m(f)^2d\mu_l
\leq \frac{\varepsilon^2}{2}.
\end{align}

For $A\in \mathcal{P}_{m,k}^l\setminus\mathcal{P}_{m,k}^{l,\infty}$ and $x\in A$ we observe 
$$
\left(\frac{1}{\mu_l(A)}\int_A (\pi_m(f)(x)-\pi_m(f)(y))\mu_l(dy)\right)^2
$$$$=\left(\sum_{i=1}^m c_i\:\frac{1}{\mu_l(A)}\int_A(\varphi_i(x)-\varphi_i(y))\mu_l(dy)\right)^2\notag
$$$$\leq \sum_{j=1}^m c_j^2 \sum_{i=1}^m \frac{1}{\mu_l(A)}\int_A(\varphi_i(x)-\varphi_i(y))^2\mu_l(dy)\notag
 \leq \frac{m}{2^{2k}}\:\left\|f\right\|_{L^2(X,\mathcal{X},\mu)}^2.\notag
$$
Here we have used Cauchy-Schwarz, Parseval's identity, Jensen's inequality and the fact that for $x,y\in A$,
$|\varphi_i(x)-\varphi_i(y)|\leq 2^{-k}$, \   $i=1,...,m$.
If $k$ is sufficiently large,
\begin{align}\label{E:est1}
\sum_{A\in\mathcal{P}_{m,k}^l\setminus\mathcal{P}_{m,k}^{l,\infty}} \int_{A}&\left(\frac{1}{\mu_l(A)}\int_A(\pi_m(f)(x)-\pi_m(f)(y))\mu_l(dy)\right)^2\mu_l(dx)\notag\\
&\leq \frac{m}{2^{2k}}\:\left\|f\right\|_{L^2(X,\mathcal{X},\mu)}^2\:\mu(X_l) <\frac{\varepsilon^2}{2}.
\end{align}
Combining (\ref{E:est1}) with (\ref{E:firstbound}) we arrive at (\ref{E:approx}).
\end{proof}

For any $k$, $l$, $m$ and $n$ consider the bounded Dirichlet forms $\mathcal{E}^{(n,m,l,k)}$ defined by
\begin{equation}\label{E:finiteDF}
\mathcal{E}^{(n,m,l,k)}(f):=2^n\left\langle\pi_{m,l,k}(f)-P_{2^{-n}}\pi_{m,l,k}(f),\pi_{m,l,k}(f)\right\rangle_{L^2(X,\mathcal{X},\mu)}, \ \ f\in L^2(X,\mathcal{X},\mu).
\end{equation}
Their generators are given by $\mathcal{L}^{(n,m,l,k)}=2^n \pi_{m,l,k}^\ast(P_{2^{-n}}-I)\pi_{m,l,k}$. From Lemma \ref{L:approxL2} we obtain
\[\lim_k\mathcal{E}^{(n,m,l,k)}(f)=\mathcal{E}^{(n,m,l)}(f)\]
for any fixed $l$, $m$ and $n$ and all $f\in L^2(X,\mathcal{X},\mu)$. Therefore we may again follow the arguments in the proof of Lemma \ref{L:sgapprox} to obtain a final approximation step.

\begin{lemma}\label{L:finite}
For any fixed $n$, $m$ and $l$ the forms $\mathcal{E}^{(n,m,l,k)}$ converge to the form $\mathcal{E}^{(n,m,l)}$ in the sense of Mosco as $k$ goes to infinity. Their generators $\mathcal{L}^{(n,m,l,k)}$ converge to the generator $\mathcal{L}^{(n,m)}$ of $\mathcal{E}^{(n,m,l)}$ in the strong resolvent sense.
\end{lemma}

In other words, any separable and bounded Dirichlet form with bounded reference measure (on a finite dimensional $L^2$-space) is the limit of essentially discrete energy forms (isomorphic, except for subset of functions of zero energy, to the finite dimensional energy form on a finite graph  as in Section \ref{S:finite}).

Combining Lemmas \ref{L:sgapprox}, \ref{L:galerkin}, \ref{L:sigma} and \ref{L:finite} we now obtain the following Corollary, which implies Theorem \ref{T:main2}.

\begin{corollary}\label{C:main2}
Let $(X,\mathcal{X},\mu)$ be a $\sigma$-finite measure space with a countably generated $\sigma$-algebra $\mathcal{X}$ and let $(\mathcal{E},\mathcal{F})$ be a Dirichlet form on $L^2(X,\mathcal{X},\mu)$ with generator $\mathcal{L}$. Let $\mathcal{E}^{(n,m,l,k)}$ be the bounded Dirichlet forms as defined in (\ref{E:finiteDF}) with generators $\mathcal{L}^{(n,m,l,k)}$. Then we have 
\[\lim_n\lim_m\lim_l\lim_k \mathcal{E}^{(n,m,l,k)}=\mathcal{E}\]
in the sense of Mosco convergence and
\[\lim_n\lim_m\lim_l\lim_k \mathcal{L}^{(n,m,l,k)}=\mathcal{L}\]
in the strong resolvent sense.
\end{corollary}

\end{document}